\documentclass[reqno]{amsart}
\usepackage{amsmath, amsthm, amssymb, amstext}

\usepackage{hyperref,xcolor}
\hypersetup{
 pdfborder={0 0 0},
 colorlinks,
}

\usepackage{enumitem}
\newtheorem{theorem}{Theorem}
\newtheorem{remark}[theorem]{Remark}

\newtheorem{proposition}[theorem]{Proposition}

\DeclareMathOperator*{\dom}{dom}              %
\DeclareMathOperator*{\divergenz}{div}              %
\DeclareMathOperator*{\ints}{int}         %
\DeclareMathOperator*{\ww}{w}         %

\newcommand{\N}{\mathbb{N}}

\newcommand{\R}{\mathbb{R}}

\newcommand{\Ss}{\textnormal{(S}_+\textnormal{)}}         %

\newcommand{\Lp}[1]{L^{#1}(\Omega)}

\newcommand{\Lprand}[1]{L^{#1}(\partial\Omega)}
\newcommand{\Wp}[1]{W^{1,#1}(\Omega)}

\newcommand{\lan}{\langle}
\newcommand{\ran}{\rangle}

\newcommand{\ph}{\varphi}

\newcommand{\into}{\int_{\Omega}}
\newcommand{\weak}{\overset{\ww}{\to}}

\newcommand{\Linf}{L^{\infty}(\Omega)}
\newcommand{\close}{\overline{\Omega}}
\newcommand{\interior}{\ints \left(C^1(\overline{\Omega})_+\right)}

\renewcommand{\l}{\left}
\renewcommand{\r}{\right}

\newcommand*\diff{\mathop{}\!\mathrm{d}}
\newcommand{\W}{W^{1,p(\cdot)}(\Omega)}

\numberwithin{theorem}{section}
\numberwithin{equation}{section}

\title[A multiplicity theorem for anisotropic Robin equations]{A multiplicity theorem for anisotropic Robin equations}

\author[N.\,S.\,Papageorgiou]{Nikolaos S.\,Papageorgiou}
\address[N.\,S.\,Papageorgiou]{National Technical University, Department of Mathematics, Zografou Campus, Athens 15780, Greece}
\email{npapg@math.ntua.gr}

\author[P.\,Winkert]{Patrick Winkert}
\address[P.\,Winkert]{Technische Universit\"{a}t Berlin, Institut f\"{u}r Mathematik, Stra\ss e des 17.\,Juni 136, 10623 Berlin, Germany}
\email{winkert@math.tu-berlin.de}


\begin{document}

\begin{abstract}
	In this paper we consider an anisotropic Robin problem driven by the $p(x)$-Laplacian and a superlinear reaction. Applying variational tools along with truncation and comparison techniques
	as well as critical groups, we prove that the problem has at least five nontrivial smooth solutions to be ordered and with sign information: two positive, two negative and the fifth nodal.
\end{abstract}

\maketitle

\section{Introduction}

Let $\Omega \subseteq \R^N$ be a bounded domain with a $C^2$-boundary $\partial \Omega$. In this paper, we study the following anisotropic Robin problem
\begin{equation}\label{problem}
	\begin{aligned}
		-\Delta_{p(\cdot)}u+\xi(x)|u|^{p(x)-2}u &= f(x,u)\quad &&\text{in } \Omega, \\
		|\nabla u|^{p(x)-2} \nabla u \cdot \nu+\beta(x)|u|^{p(x)-2}u&=0&&\text{on }\partial\Omega,
	\end{aligned}
\end{equation}
where  $\nu(x)$ denotes the outer unit normal at $x\in \partial\Omega$, $\beta\in C^{0,\alpha}(\partial\Omega)$ with $\alpha \in (0,1)$, $\beta\geq 0$ and for $p \in C^{0,1}(\close)$ with $1<\min_{x\in\close}p(x)$ we denote by $\Delta_{p(\cdot)}$ the $p(x)$-Laplacian which is given by
\begin{align*}
	\Delta_{p(\cdot)} u = \divergenz \l(|\nabla u|^{p(x)-2} \nabla u \r) \quad\text{for all }u \in \Wp{p(\cdot)}.
\end{align*}
In the left-hand side of \eqref{problem} there is also a potential term $\xi(x)|u|^{p(x)-2}u$ with $\xi \in \Linf$ and $\xi\geq 0$. In the right-hand side of \eqref{problem} there is a Carath\'eodory function $f\colon \Omega\times\R\to\R$, that is, $x\to f(x,s)$ is measurable for all $s\in\R$ and $s\to f(x,s)$ is continuous for a.\,a.\,$x\in\Omega$. We suppose that $f(x,\cdot)$ is $(p_+-1)$-superlinear as $s\to \pm\infty$ but without assuming the usual Ambrosetti-Rabinowitz condition, where $p_+=\max_{x\in\close}p(x)$. Near zero $f(x,\cdot)$ exhibits an oscillatory behavior. 

Using variational tools from the critical point theory along with appropriate truncation and comparison techniques, we prove the existence of at least five nontrivial smooth solutions, all with sign information and ordered. 

Elliptic equations driven by the anisotropic Dirichlet $p$-Laplacian have been studied extensively in the last decade. The books of Diening-Harjulehto-H\"{a}st\"{o}-R$\mathring{\text{u}}$\v{z}i\v{c}ka \cite{3-Diening-Harjulehto-Hasto-Ruzicka-2011} and R\u{a}dulescu-Repov\v{s} \cite{15-Radulescu-Repovs-2015} contain a rich bibliography on the subject. In contrast, the study of
anisotropic Robin problems is lagging behind. Deng \cite{1-Deng-2009} studied the Robin problem
\begin{equation}\label{problem2}
	\begin{aligned}
		-\Delta_{p(\cdot)}u&= \lambda f(x,u)\quad &&\text{in } \Omega, \\
		|\nabla u|^{p(x)-2} \nabla u \cdot \nu+\beta(x)|u|^{p(x)-2}u&=0&&\text{on }\partial\Omega,
	\end{aligned}
\end{equation}
and proved the existence  of two positive solutions of problem \eqref{problem2} when $p\in C^1(\close)$ and under the Ambrosetti-Rabinowitz condition. A similar problem under the same assumptions as in \cite{1-Deng-2009} was treated by Fan-Deng \cite{7-Fan-Deng-2009}, namely
\begin{equation}\label{problem3}
	\begin{aligned}
		-\Delta_{p(\cdot)}u+\lambda |u|^{p(x)-2}u&= f(x,u)\quad &&\text{in } \Omega, \\
		|\nabla u|^{p(x)-2} \nabla u \cdot \nu&=\ph&&\text{on }\partial\Omega.
	\end{aligned}
\end{equation}
Only positive solutions for \eqref{problem3} is shown but no sign-changing solution is obtained. In 2010, Deng-Wang \cite{2-Deng-Wang-2010} considered existence and nonexistence of a nonhomogeneous Neumann problem given by
\begin{equation}\label{problem4}
	\begin{aligned}
		-\Delta_{p(\cdot)}u+\lambda |u|^{p(x)-2}u&= f(x,u)\quad &&\text{in } \Omega, \\
		|\nabla u|^{p(x)-2} \nabla u \cdot \nu&=g(x,u)&&\text{on }\partial\Omega.
	\end{aligned}
\end{equation}
It is proved that there exists a parameter $\lambda^*>0$ such that problem \eqref{problem4} has at least two positive solutions for all $\lambda>\lambda^*$. We also  mention the works of Gasi\'nski-Papageorgiou \cite{8-Gasinski-Papageorgiou-2011}, Papageorgiou-R\u{a}dulescu-Tang \cite{14-Papageorgiou-Radulescu-Tang-2021} and Wang-Fan-Ge \cite{17-Wang-Fan-Ge-2009}. Except for \cite{8-Gasinski-Papageorgiou-2011}, the above mentioned works consider parametric equations and focus on the existence and multiplicity of positive solutions. Gasi\'nski-Papageorgiou \cite{8-Gasinski-Papageorgiou-2011} considered the Neumann problem 
\begin{equation*}
	\begin{aligned}
		-\Delta_{p(\cdot)}u&= f(x,u)\quad &&\text{in } \Omega, \\
		|\nabla u|^{p(x)-2} \nabla u \cdot \nu&=0&&\text{on }\partial\Omega.
	\end{aligned}
\end{equation*}
and prove the existence of three nontrivial smooth solutions but they do not produce nodal solutions. The novelties in our work in contrast to the above mentioned papers can be summarized as follows:
\begin{enumerate}[leftmargin=0.8cm]
	\item[$\bullet$]
		We only need $p$ to be Lipschitz continuous.
	\item[$\bullet$]
		We do not need to assume the Ambrosetti-Rabinowitz condition. We can weaken the assumptions, see \textnormal{H$_1$}\textnormal{(ii)}, \textnormal{(iii)} in Section \ref{section_2}.
	\item[$\bullet$]
		We obtain not only constant sign solutions, but also a sign-changing solution.
	\item[$\bullet$]
		All the solutions we obtain are ordered with concrete sign information.
\end{enumerate}

Finally we mention the works of Deng \cite{Deng-2008} and Deng-Wang-Cheng \cite{Deng-Wang-Cheng-2011} concerning the Steklov and Robin eigenvalue problems of the anisotropic $p$-Laplacian, respectively.\\

The paper is organized as follows. In Section \ref{section_2} we recall the basic properties of the variable exponent Sobolev spaces and the anisotropic $p$-Laplacian, mention some tools/definitions we need later (Cerami-condition, critical groups) and state the main hypotheses on the data of our problem. Section \ref{section_3} deals with the existence of constant sign solutions. The first pair of positive and negative solutions is obtained in Proposition \ref{prop_4} by using the direct method of calculus and the existence of the second pair of positive and negative solutions, stated in Proposition \ref{prop_5}, is proved via the mountain pass theorem. The rest of the section is devoted to the existence of extremal constant sign solutions, see Proposition \ref{prop_8}, which are needed later in order to find a sign-changing solution. Finally, Section \ref{section_4} is concerned with the existence of a nodal solution to problem \eqref{problem} which lies between the extremal constant sign solutions. This result is stated in Proposition \ref{prop_9} and the proof relies on the combination of the mountain pass theorem and critical groups. The full multiplicity result is given at the end in Theorem \ref{main_result}.

\section{Preliminaries and Hypotheses}\label{section_2}

The study of problem \eqref{problem} uses function spaces with variable exponents. A comprehensive introduction on the subject can be found in the book of Diening-Harjulehto-H\"{a}st\"{o}-R$\mathring{\text{u}}$\v{z}i\v{c}ka \cite{3-Diening-Harjulehto-Hasto-Ruzicka-2011}.

In what follows we denote by $M(\Omega)$ the vector space of functions $u\colon \Omega\to\R$ which are measurable. As usual, we identify two such functions when they differ only on a Lebesgue-null set. Given $r \in C(\close)$ we define
\begin{align*}
	r_-=\min_{x\in \close}r(x) \quad\text{and}\quad r_+=\max_{x\in\close} r(x)
\end{align*}
and introduce the set 
\begin{align*}
	E_1=\l \{r \in C(\close)\, : \, 1<r_- \r\}.
\end{align*}
Then, for $r \in E_1$, we introduce the variable exponent Lebesgue space $\Lp{r(\cdot)}$ defined by
\begin{align*}
	\Lp{r(\cdot)}=\l\{u \in M(\Omega)\,:\, \into |u|^{r(x)} \diff x<\infty \r\}.
\end{align*}
We equip this space with the Luxemburg norm defined by
\begin{align*}
	\|u\|_{r(\cdot)} =\inf \l \{\lambda>0 \, : \, \into \l(\frac{|u|}{\lambda}\r)^{r(x)}\diff x \leq 1 \r\}.
\end{align*}
Then $\Lp{r(\cdot)}$ is a separable, reflexive Banach space.

Moreover, we denote by $r'(x)=\frac{r(x)}{r(x)-1}$ the conjugate variable exponent to $r\in E_1$, that is,
\begin{align*}
	\frac{1}{r(x)}+\frac{1}{r'(x)}=1 \quad\text{for all }x\in\close.
\end{align*}
It is clear that $r'\in E_1$. We know that $\Lp{r(\cdot)}^*=\Lp{r'(\cdot)}$ and the following version of H\"older's inequality holds
\begin{align*}
	\into |uv| \diff x \leq \l[\frac{1}{r_-}+\frac{1}{r'_-}\r] \|u\|_{r(\cdot)}\|v\|_{r'(\cdot)}
\end{align*}
for all $u\in \Lp{r(\cdot)}$ and for all $v \in \Lp{r'(\cdot)}$.

On the boundary $\partial\Omega$ we consider the $(N-1)$-dimensional Hausdorff (surface) measure $\sigma$. Using this measure we can define the boundary variable exponent Lebesgue spaces $\Lprand{r(\cdot)}$ for $r \in E_1$.

The corresponding variable exponent Sobolev spaces can be defined in a natural way using the variable exponent Lebesgue spaces. So, given $r \in E_1$, we define 
\begin{align*}
	\Wp{r(\cdot)}=\l\{ u \in \Lp{r(\cdot)} \,:\, |\nabla u| \in \Lp{r(\cdot)}\r\}
\end{align*}
with $\nabla u$ being the gradient of $u\colon\Omega\to\R$. This space is equipped with the norm
\begin{align*}
	\|u\|_{1,r(\cdot)}=\|u\|_{r(\cdot)}+\|\nabla u\|_{r(\cdot)} \quad\text{for all } u \in \Wp{r(\cdot)}
\end{align*}
with $\|\nabla u\|_{r(\cdot)}=\|\,|\nabla u|\,\|_{r(\cdot)}$. The space $\Wp{r(\cdot)}$ is a separable and reflexive Banach space. 

For $r \in E_1$ we introduce the critical Sobolev variable exponents $r^*$ and $r_*$ defined by
\begin{align*}
	r^*(x)=
	\begin{cases}
		\frac{Nr(x)}{N-r(x)} & \text{if }r(x)<N,\\
		\ell_1(x)& \text{if } N \leq r(x),
	\end{cases} \quad\text{for all }x\in\close.
\end{align*}
and
\begin{align*}
	r_*(x)=
	\begin{cases}
		\frac{(N-1)r(x)}{N-r(x)} & \text{if }r(x)<N,\\
		\ell_2(x)& \text{if } N \leq r(x),
	\end{cases} \quad\text{for all }x\in\partial\Omega,
\end{align*}
where $\ell_1 \in C(\close)$, $\ell_2 \in C(\partial\Omega)$ are arbitrarily chosen such that $r(x)<\ell_1(x)$ for all $x \in \close$ and $r(x)<\ell_2(x)$ for all $x \in \partial\Omega$.

Suppose that $r \in C^{0,1}(\close)\cap E_1$ and $q \in C(\close)$ with $1 \leq q_-$. Then we have the following anisotropic Sobolev embeddings
\begin{align*}
	&\Wp{r(\cdot)}\hookrightarrow \Lp{q(\cdot)} \quad \text{continuously if } q(x) \leq r^*(x) \text{ for all }x\in\close,\\
	&\Wp{r(\cdot)}\hookrightarrow \Lp{q(\cdot)} \quad \text{compactly if } q(x) < r^*(x) \text{ for all }x\in\close.
\end{align*}

Similarly, if $r \in C^{0,1}(\close)\cap E_1$ and $q \in C(\partial\Omega)$ with $1 \leq q_-$, then we have the anisotropic trace embeddings
\begin{align*}
	&\Wp{r(\cdot)}\hookrightarrow \Lprand{q(\cdot)} \quad \text{continuously if } q(x) \leq r_*(x) \text{ for all }x\in\close,\\
	&\Wp{r(\cdot)}\hookrightarrow \Lprand{q(\cdot)} \quad \text{compactly if } q(x) < r_*(x) \text{ for all }x\in\close.
\end{align*}
We refer to Diening-Harjulehto-H\"{a}st\"{o}-R$\mathring{\text{u}}$\v{z}i\v{c}ka \cite{3-Diening-Harjulehto-Hasto-Ruzicka-2011} and Fan \cite{5-Fan-2008}.

In the study of these variable exponent spaces, the following modular function is useful
\begin{align*}
	\varrho_{r(\cdot)}(u) =\into |u|^{r(x)}\diff x \quad\text{for all } u\in\Lp{r(\cdot)}.
\end{align*}
For $u \in \Wp{r(\cdot)}$ we write $\varrho_{r(\cdot)}(\nabla u)=\varrho_{r(\cdot)}(|\nabla u|)$.

The following proposition illustrates the relation between this modular and the Luxemburg norm.

\begin{proposition}\label{prop_1}
	Let $r\in E_1$, let $u \in\Lp{r(\cdot)}$ and let $\{u_n\}_{n\in\N}\subseteq \Lp{r(\cdot)}$. The following assertions hold:
	\begin{enumerate}
		\item[\textnormal{(i)}]
		$\|u\|_{r(\cdot)}=\eta \quad\Longleftrightarrow\quad \varrho_{r(\cdot)}\l(\frac{u}{\eta}\r)=1$;
		\item[\textnormal{(ii)}]
		$\|u\|_{r(\cdot)}<1$ (resp. $=1$, $>1$) $\quad\Longleftrightarrow\quad \varrho_{r(\cdot)}(u)<1$ (resp. $=1$, $>1$);
		\item[\textnormal{(iii)}]
		$\|u\|_{r(\cdot)}\leq 1$ $\quad\Longrightarrow\quad$ $\|u\|_{r(\cdot)}^{r_+} \leq \varrho_{r(\cdot)}(u) \leq \|u\|_{r(\cdot)}^{r_-}$;\\
		$\|u\|_{r(\cdot)}\geq 1$ $\quad\Longrightarrow\quad$ $\|u\|_{r(\cdot)}^{r_-} \leq \varrho_{r(\cdot)}(u) \leq \|u\|_{r(\cdot)}^{r_+}$;
		\item[\textnormal{(iv)}]
		$\|u_n\|_{r(\cdot)} \to 0$  $\quad\Longleftrightarrow\quad\varrho_{r(\cdot)}(u_n)\to 0$;
		\item[\textnormal{(v)}]
		$\|u_n\|_{r(\cdot)} \to \infty$ $\quad\Longleftrightarrow\quad\varrho_{r(\cdot)}(u_n)\to \infty$;
	\end{enumerate}
\end{proposition}

Let $A_{r(\cdot)}\colon \Wp{r(\cdot)}\to \Wp{r(\cdot)}^*$ be the nonlinear operator defined by
\begin{align*}
	\l\lan A_{r(\cdot)}(u),h\r\ran= \into |\nabla u|^{r(x)-2} \nabla u \cdot \nabla h\diff x\quad\text{for all }u,h\in\Wp{r(\cdot)}.
\end{align*} 
This operator has the following properties, see, for example Gasi\'nski-Papageorgiou \cite{8-Gasinski-Papageorgiou-2011} and R\u{a}dulescu-Repov\v{s} \cite[p.\,40]{15-Radulescu-Repovs-2015}. 

\begin{proposition}\label{prop_2}
	The operator $A_{r(\cdot)}\colon \Wp{r(\cdot)}\to \Wp{r(\cdot)}^* $ is bounded (so it maps bounded sets to bounded sets), continuous, strictly monotone (which implies it is also maximal monotone) and of type $\Ss$, that is,
	\begin{align*}
		u_n\weak u \quad\text{in }\Wp{r(\cdot)} \quad\text{and}\quad \limsup_{n\to\infty} \l\lan A_{r(\cdot)}(u_n),u_n-u\r\ran \leq 0
	\end{align*}
	imply $u_n\to u$ in $\Wp{r(\cdot)}$.
\end{proposition}

In the anisotropic regularity theory we need the Banach space $C^1(\close)$. This is an ordered Banach space with positive order cone
\begin{align*}
	C^1(\overline{\Omega})_+=\left\{u \in C^1(\overline{\Omega})\,:\, u(x) \geq 0 \text{ for all } x \in \close\right\}.
\end{align*}
This cone has a nonempty interior given by
\begin{align*}
	\ints \left(C^1(\overline{\Omega})_+\right)=\left\{u \in C^1(\overline{\Omega})_+\,:\, u(x)>0 \text{ for all } x \in \close  \right\}.
\end{align*}
We will also use another open cone in $C^1(\overline{\Omega})$ defined by
\begin{align*}
	D_+=\l \{u\in C^1(\close)\,:\,u(x)>0 \text{ for all }x\in\Omega \text{ and }  \frac{\partial u}{\partial \nu}\Big|_{\partial\Omega \cap u^{-1}(0)}<0 \r\}
\end{align*}
where $\frac{\partial u}{\partial \nu}=\nabla u \cdot \nu$.

Given $u\in \Wp{r(\cdot)}$, we set $u^{\pm}=\max \{\pm u,0\}$ being the positive and negative part of $u$, respectively. We know that $u=u^+-u^-$, $|u|=u^++u^-$ and $u^{\pm}\in\Wp{r(\cdot)}$. If $u,v\colon \Omega \to \R$ are measurable functions and $u(x)\leq v(x)$ for a.\,a.\,$x\in\Omega$, then we introduce the following order interval in $\Wp{r(\cdot)}$
\begin{align*}
	[u,v]&=\left\{h\in\Wp{r(\cdot)}: u(x) \leq h(x) \leq v(x) \text{ for a.\,a.\,}x\in\Omega\right\}.
\end{align*}
Moreover, we denote by $\sideset{}{_{C^1(\close)}} \ints [u,v]$  the interior of $[u,v]\cap C^1(\close)$ in $C^1(\close)$. Furthermore, we define
\begin{align*}
	[u) & = \left\{h\in \Wp{r(\cdot)}: u(x) \leq h(x) \text{ for a.\,a.\,}x\in\Omega\right\}.
\end{align*}

Suppose that $X$ is a Banach space and $\ph\in C^1(X)$. We introduce the following sets 
\begin{align*}
	K_\ph&=\left\{ u\in X\,:\, \ph'(u)=0\right\},\\
	\ph^c&=\l\{u\in X\,:\,\ph(u)\leq c\r\}\quad\text{with }c\in \R.
\end{align*}

We say that $\ph$ satisfies the ``Cerami condition'', C-condition for short, if every sequence $\{u_n\}_{n\in\N}\subseteq X$ such that $\{\ph(u_n)\}_{n\in\N}\subseteq \R$ is bounded and 
\begin{align*}
	\left(1+\|u_n\|_X\right)\ph'(u_n) \to 0\quad\text{in }X^* \quad\text{as }n\to \infty,
\end{align*}    
admits a strongly convergent subsequence. 

If $Y_2 \subseteq Y_1 \subseteq X$, then we denote by $H_k(Y_1,Y_2)$ with $k \in \N_0$, the $k\overset{\text{th}}{=}$-relative singular homology group with integer coefficients. If $u \in K_\ph$ is isolated, then the $k\overset{\text{th}}{=}$-critical group of $\ph$ at $u$ is defined by
\begin{align*}
	C_k(\ph,u)=H_k\l(\ph^c\cap U, \l(\ph^c\cap U\r)\setminus \{u\}\r)\quad\text{for all }k \in \N_0
\end{align*}
with $c=\ph(u)$ and a neighborhood $U$ of $u$ such that $\ph^c\cap K_\ph\cap U=\{u\}$. The excision property of singular homology implies that this definition of critical groups is independent of the isolating neighborhood $U$.

Now we introduce our hypotheses on the exponent $p(\cdot)$, the potential $\xi(\cdot)$ and the boundary coefficient $\beta(\cdot)$:
\begin{enumerate}
	\item[\textnormal{H$_0$}:]
	$p\in C^{0,1}(\close)$, $1<p_-(x)=\min_{x\in\close}p(x)<N$, $\xi \in \Linf$, $\beta \in C^{0,\alpha}(\partial\Omega)$ with $\alpha \in (0,1)$, $\xi(x) \geq 0$ for a.\,a.\,$x\in\Omega$, $\beta(x) \geq 0$ for all $x \in \partial\Omega$ and $\xi \not\equiv 0$ or $\beta\not\equiv 0$.
\end{enumerate}

Note that the case $\beta=0$ is also included and corresponds to the Neumann problem.

We introduce the $C^1$-functional $\gamma_{p(\cdot)}\colon\W\to \R$ defined by
\begin{align*}
	\gamma_{p(\cdot)}(u)=\into \frac{1}{p(x)}|\nabla u|^{p(x)}\diff x+\into \frac{\xi(x)}{p(x)} |u|^{p(x)}\diff x+\int_{\partial\Omega} \frac{\beta(x)}{p(x)}|u|^{p(x)}\diff \sigma
\end{align*}
for all $u \in \W$. We have
\begin{align*}
	\l\langle\gamma_{p(\cdot)}'(u),h \r\rangle
	=\l\langle A_{p(\cdot)}(u),h\r\rangle+\into \xi(x)|u|^{p(x)-2}uh\diff x+\int_{\partial\Omega}\beta(x)|u|^{p(x)-2}uh\diff \sigma
\end{align*}
for all $u,h\in \W$. Moreover, let $\varrho_{0}\colon \W\to \R$ be the modular function defined by
\begin{align*}
	\varrho_{0}(u)=\varrho_{p(\cdot)}(\nabla u)+\into \xi(x)|u|^{p(x)}\diff x+\int_{\partial\Omega}\beta(x)|u|^{p(x)}\diff \sigma
\end{align*}
for all $u \in \W$. 

In the sequel we denote by $\|\,\cdot\,\|$ the norm of the Sobolev space $\W$ defined by
\begin{align*}
	\|u\|=\|u\|_{p(\cdot)}+\|\nabla u\|_{p(\cdot)}\quad\text{for }u\in \W.
\end{align*}

The following estimates for $\gamma_{p(\cdot)}(\cdot)$ will be useful in what follows. The result can be found in the recent work of Papageorgiou-R\u{a}dulescu-Tang \cite{14-Papageorgiou-Radulescu-Tang-2021}.

\begin{proposition}\label{prop_3}
	If hypotheses \textnormal{H$_0$} hold, then there exist $\hat{c}_0, \hat{c}>0$ such that
	\begin{align*}
		\hat{c} \|u\|^{p_+} \leq \frac{1}{p_+} \varrho_{0}(u) \leq \gamma_{p(\cdot)}(u) \leq \frac{1}{p_-} \varrho_{0}(u) \leq \hat{c}_0 \|u\|^{p_-}\qquad &\text{if } \|u\| \leq 1,\\
		\hat{c} \|u\|^{p_-} \leq \frac{1}{p_+} \varrho_{0}(u) \leq \gamma_{p(\cdot)}(u) \leq \frac{1}{p_-} \varrho_{0}(u) \leq \hat{c}_0 \|u\|^{p_+}\qquad &\text{if } \|u\| \geq 1.
	\end{align*}
\end{proposition}

Now we are ready to state our hypotheses on the nonlinearity $f\colon\Omega\times\R\to\R$.
\begin{enumerate}
	\item[\textnormal{H$_1$}:]
		$f\colon\Omega\times\R\to\R$ is a Carath\'eodory function such that $f(x,0)=0$ for a.\,a.\,$x\in\Omega$ and
		\begin{enumerate}
			\item[\textnormal{(i)}]
				there exists $a \in \Linf$ such that
				\begin{align*}
					|f(x,s)| \leq a(x) \l [ 1+|s|^{r(x)-1}\r]
				\end{align*}
				for a.\,a.\,$x\in \Omega$, for all $s\in\R$ with $r \in \Linf$ such that $p_+<r(x)<p^*(x)$ for a.\,a.\,$x\in\Omega$;
			\item[\textnormal{(ii)}]
				if $F(x,s)=\displaystyle\int_0^s f(x,t)\diff t$, then
				\begin{align*}
					\lim_{s\to \pm\infty} \frac{F(x,s)}{|s|^{p_+}}=+\infty \quad\text{uniformly for a.\,a.\,}x\in\Omega;
				\end{align*}
			\item[\textnormal{(iii)}]
				there exists a function $q\in C(\close)$ such that
				\begin{align*}
					q(x) \in \l (\l(r_+-p_-\r)\frac{N}{p_-}, p^*(x) \r)\quad\text{ for all }x\in\close
				\end{align*}
				and
				\begin{align*}
					0<\eta \leq \liminf_{s\to +\infty} \frac{f(x,s)s-p_+F(x,s)}{|s|^{q(x)}}
				\end{align*}
				uniformly for a.\,a.\,$x\in\Omega$;
			\item[\textnormal{(iv)}]
				there exist $\eta_-<0<\eta_+$, $\tau\in C(\close)$ and $\delta>0$ such that
				\begin{align*}
					f\l(x,\eta_+\r) \leq -c_0<0<c_1\leq f\l(x,\eta_-\r)
				\end{align*}
				for a.\,a.\,$x\in\Omega$, $\tau_+<p_-$ and
				\begin{align*}
					f(x,s)s \geq c_2 |s|^{\tau(x)}
				\end{align*}
				for a.\,a.\,$x\in\Omega$, for all $s \in \R$ and for some $c_2>0$;
			\item[\textnormal{(v)}]
				there exists $\hat{\xi}>0$ such that 
				\begin{align*}
					s\to f(x,s)+\hat{\xi} |s|^{p(x)-2}s
				\end{align*}
				is nondecreasing on $[\eta_-,\eta_+]$ for a.\,a.\,$x\in \Omega$.
		\end{enumerate}
\end{enumerate}

\begin{remark}
	Hypotheses \textnormal{H$_1$}\textnormal{(ii)}, \textnormal{(iii)} imply that $f(x,\cdot)$ is $(p_+-1)$-superlinear for a.\,a.\,$x\in\Omega$. However, we do not use the Ambrosetti-Rabinowitz condition as it was done in most previous works on the subject, see Deng \cite{1-Deng-2009}, Deng-Wang \cite{2-Deng-Wang-2010} and Fan-Deng \cite{7-Fan-Deng-2009}, for example. Hypothesis \textnormal{H$_1$}\textnormal{(iv)} dictates an oscillatory behavior near zero. The following function satisfies hypotheses \textnormal{H$_1$}, but fails to fulfill the Ambrosetti-Rabinowitz condition, where we drop the $x$-dependence for simplification:
	\begin{align*}
		f(s)=
		\begin{cases}
			|s|^{\tau(x)-2}s-2|s|^{\mu(x)-2}s&\text{if } |s|\leq 1,\\[1ex]
			|s|^{p_+-2}s\ln(|s|)-|s|^{q(x)-2}s&\text{if } 1<|s|,
		\end{cases}
	\end{align*}
	with $\tau\in E_1$, $\mu,q\in \Linf$ and $q(x)\leq p_+$ for a.\,a.\,$x\in\Omega$. Note that $f$ fails to satisfy the requirements in \cite{1-Deng-2009}, \cite{2-Deng-Wang-2010} and \cite{7-Fan-Deng-2009}.
\end{remark}

\section{Constant sign solutions}\label{section_3}

We start by producing two localized constant sign solutions. To do this, we do not need the complete set of hypotheses \textnormal{H$_1$}. More precisely we do not need the asymptotic conditions as $s \to \pm\infty$.

\begin{proposition}\label{prop_4}
	If hypotheses \textnormal{H$_0$}, \textnormal{H$_1$}\textnormal{(i)}, \textnormal{(iv)}, \textnormal{(v)} hold, then problem \eqref{problem} has two constant sign solutions 
	\begin{align*}
		u_0\in\interior\quad\text{and}\quad v_0\in-\interior
	\end{align*}
	such that
	\begin{align*}
		\eta_-<v_0(x)<0<u_0(x)<\eta_+\quad\text{for all }x \in \close.
	\end{align*}
\end{proposition}

\begin{proof}
	First we show the existence of the positive solution. To this end, we introduce the Carath\'eodory function $\hat{f}_+\colon\Omega\times\R\to\R$ defined by
	\begin{align}\label{2}
		\hat{f}_+(x,s)=
		\begin{cases}
			f\l(x,s^+\r) &\text{if } s \leq \eta_+,\\[1ex]
			f\l(x,\eta_+\r) & \text{if } \eta_+<s.
		\end{cases}
	\end{align}
	We set $\hat{F}_+(x,s)=\int^s_0 \hat{f}_+(x,t)\diff t$ and consider the $C^1$-functional $\hat{\psi}_+\colon \W\to\R$ defined by
	\begin{align*}
		\hat{\psi}_+(u)=\gamma_{p(\cdot)}(u)-\into \hat{F}_+(x,u)\diff x\quad\text{for all }u\in\W.
	\end{align*}
	
	From the truncation in \eqref{2} and Proposition \ref{prop_3} it is clear that $\hat{\psi}_+\colon \W\to\R$ is coercive. Moreover, the anisotropic Sobolev embedding theorem and the compactness of the anisotropic trace map imply that $\hat{\psi}_+\colon \W\to\R$ is also sequentially weakly lower semicontinuous. So, by the Weierstra\ss-Tonelli theorem we can find $u_0\in\W$ such that
	\begin{align}\label{3}
		\hat{\psi}_+\l(u_0\r)=\min \l[\hat{\psi}_+(u)\,:\, u \in \W\r].
	\end{align}
	
	Let $u \in \interior$ and choose $t \in (0,1)$ small enough such that
	\begin{align*}
		0<tu(x)\leq \min \l\{\eta_+,\delta\r\}\quad\text{for all }x \in \close,
	\end{align*}
	see hypothesis \textnormal{H$_1$}\textnormal{(iv)}. Applying hypothesis \textnormal{H$_1$}\textnormal{(iv)} and recalling that $t\in(0,1)$, we have
	\begin{align*}
		\hat{\psi}_+(tu) \leq \frac{t^{p_-}}{p_-} \varrho_{0}(u)-\frac{t^{\tau_+}}{\tau_+}c_0\varrho_{\tau(\cdot)}(u).
	\end{align*}
	Since $\tau_+<p_-$ we can choose $t \in (0,1)$ sufficiently small such that $\hat{\psi}_+(tu)<0$. Hence, since $u_0$ is the global minimizer of $\hat{\psi}_+$, see \eqref{3}, we know that
	\begin{align*}
		\hat{\psi}_+(u_0)<0=\hat{\psi}_+(0).
	\end{align*}
	Thus, $u_0\neq 0$.
	
	From \eqref{3} we have $(\hat{\psi}_+)'(u_0)=0$ which is equivalent to
	\begin{align}\label{4}
		\l\langle \gamma_{p(\cdot)}'(u_0),h\r\rangle=\into \hat{f}_+(x,u_0)h\diff x\quad\text{for all }h \in \W.
	\end{align}
	Choosing $h=-u_0^- \in \W$ in \eqref{4} and using \eqref{2} gives
	\begin{align*}
		\varrho_{0}(u_0^-)=0.
	\end{align*}
	Hence, from Proposition \ref{prop_3}, we get $u_0 \geq 0$ with $u_0\neq 0$.
	
	Next, we choose $h=\l(u_0-\eta_+\r)^+\in\W$ in \eqref{4}. Applying the definition of the truncation in \eqref{2} and hypothesis \textnormal{H$_1$}\textnormal{(iv)} we obtain
	\begin{align*}
		\l\langle \gamma_{p(\cdot)}'(u_0),\l(u_0-\eta_+\r)^+\r\rangle
		&=\into f\l(x,\eta_+\r)\l(u_0-\eta_+\r)^+\diff x\\
		& \leq 0=\l\langle \gamma_{p(\cdot)}'(\eta_+),\l(u_0-\eta_+\r)^+\r\rangle.
	\end{align*}
	So, $u_0 \leq \eta_+$, see hypothesis \textnormal{H$_0$}.

	We have proved that 
	\begin{align}\label{5}
		u_0\in \l[0,\eta_+\r], \quad u_0\neq 0.
	\end{align}
	Then, \eqref{5}, \eqref{2} and \eqref{4} imply that $u_0$ is a positive solution of problem \eqref{problem}. From the anisotropic regularity theory, see Fan \cite[Theorem 1.3]{6-Fan-2007}, we have $u_0\in C^1(\close)_+\setminus\{0\}$. Finally the anisotropic maximum principle of Zhang \cite[Theorem 1.2]{19-Zhang-2005} implies that $u_0\in\interior$.
	
	Let $\hat{\xi}>0$ be as given in hypothesis \textnormal{H$_1$}\textnormal{(v)}. Then, by using \eqref{5} and hypothesis \textnormal{H$_1$}\textnormal{(v)} one gets
	\begin{align*}
		-\Delta_{p(\cdot)}u_0+ \l(\xi(x)+\hat{\xi}\r)u_0^{p(x)-1}
		& =f(x,u_0)+\hat{\xi}u_0^{p(x)-1}\\
		& \leq f\l(x,\eta_+\r)+\hat{\xi}\eta_+^{p(x)-1}\\
		&\leq -c_0+\hat{\xi}\eta_+^{p(x)-1}\\
		&\leq -\Delta_{p(\cdot)} \eta_++\l(\xi(x)+\hat{\xi}\r)\eta_+^{p(x)-1}\quad \text{in }\Omega.
	\end{align*}
	From Proposition 2.5 of Papageorgiou-R\u{a}dulescu-Repov\v{s} \cite{12-Papageorgiou-Radulescu-Repovs-2020} we then conclude that $\eta_+-u_0 \in D_+$.
	
	Similarly, using the Carath\'eodory function $\hat{f}_-\colon\Omega\times\R\to\R$ defined by
	\begin{align*}
		\hat{f}_-(x,s)=
		\begin{cases}
			f\l(x,\eta_-\r) &\text{if } s < \eta_-,\\[1ex]
			f\l(x,s\r) & \text{if } \eta_-\leq s
		\end{cases}
	\end{align*}
	and reasoning as above, we produce a negative solution
	\begin{align*}
		v_0 \in - \interior \quad\text{and}\quad v_0-\eta_- \in D_+.
	\end{align*}
\end{proof}

From Proposition \ref{prop_4} it follows that
\begin{align}\label{6}
	u_0 \in \sideset{}{_{C^1(\close)}} \ints \l[0,\eta_+\r]\quad\text{and}\quad
	v_0 \in \sideset{}{_{C^1(\close)}} \ints \l[\eta_-,0\r].
\end{align}

Now, using these localized constant sign solutions, we are going to show the existence of two more such solutions, one is larger than $u_0$ and the other one is smaller than $v_0$. So, we will have four smooth constant sign solutions which are ordered. For this we will use the asymptotic conditions as $s \to \pm\infty$. 

\begin{proposition}\label{prop_5}
	If hypotheses \textnormal{H$_0$}, \textnormal{H$_1$} hold, then problem \eqref{problem} has two more constant sign solutions 
	\begin{align*}
		\hat{u}\in\interior\quad\text{and}\quad \hat{v}\in-\interior
	\end{align*}
	such that
	\begin{align*}
		\hat{u}\neq u_0, \ u_0 \leq \hat{u}\quad\text{and}\quad \hat{v}\neq v_0, \ \hat{v}\leq v_0.
	\end{align*}
\end{proposition}

\begin{proof}
	We start with the existence of a second positive solution. To this end, we introduce the Carath\'eodory function $g_+\colon\Omega\times\R\to\R$ defined by
	\begin{align}\label{7}
		g_+(x,s)=
		\begin{cases}
			f\l(x,u_0(x)\r) &\text{if } s \leq u_0(x),\\[1ex]
			f\l(x,s\r) & \text{if } u_0(x) <s.
		\end{cases}
	\end{align}
	Moreover, we will use the truncation of $g_+(x,\cdot)$ at $\eta_+$, recall that $u_0(x)<\eta_+$ for all $x \in \Omega$. So we introduce the Carath\'eodory function $\hat{g}_+\colon\Omega\times\R\to\R$ defined by
	\begin{align}\label{8}
		\hat{g}_+(x,s)=
		\begin{cases}
			g\l(x,s\r) &\text{if } s \leq \eta_+,\\[1ex]
			g\l(x,\eta_+\r) & \text{if } \eta_+ <s.
		\end{cases}
	\end{align}
	We set $G_+(x,s)=\int^s_0 g_+(x,t)\diff t$, $\hat{G}_+(x,s)=\int^s_0 \hat{g}_+(x,t)\diff t$ and consider the $C^1$-functionals $\sigma_+, \hat{\sigma}_+\colon \W\to\R$ defined by
	\begin{align*}
		\sigma_+(u)&=\gamma_{p(\cdot)}(u)-\into G_+(x,u)\diff x\quad\text{for all }u \in \W,\\
		\hat{\sigma}_+(u)&=\gamma_{p(\cdot)}(u)-\into \hat{G}_+(x,u)\diff x\quad\text{for all }u \in \W.
	\end{align*}
	
	Using \eqref{7}, \eqref{8} and the anisotropic regularity theory, see Winkert-Zacher \cite{18-Winkert-Zacher-2012} (see also Ho-Kim-Winkert-Zhang \cite{Ho-Kim-Winkert-Zhang-2021}) and Fan \cite{6-Fan-2007}, we have
	\begin{align}\label{9}
		K_{\sigma_+} \subseteq [u_0)\cap \interior\quad\text{and}\quad K_{\hat{\sigma}_+}\subseteq [u_0,\eta_+]\cap \interior.
	\end{align}
	Moreover, it is clear that from \eqref{7} and \eqref{8} we know that
	\begin{align}\label{10}
		\sigma_+\big|_{[0,\eta_+]}=\hat{\sigma}_+\big|_{[0,\eta_+]}.
	\end{align}
	From \eqref{9} we see that we can always assume that
	\begin{align}\label{11}
		K_{\hat{\sigma}_+}=\{u_0\}.
	\end{align}
	Otherwise, we would infer from \eqref{9} and \eqref{8} that we already have a second positive smooth solution of \eqref{problem} larger than $u_0$ and so we are done.
	
	From \eqref{8} and Proposition \ref{prop_3} it is clear that $\hat{\sigma}_+\colon \W\to\R$ is coercive. Also it is sequentially weakly lower semicontinuous. Hence, its global minimizer exists, that is, we find  $\tilde{u}_0\in\W$ such that
	\begin{align*}
		\hat{\sigma}_+\l(\tilde{u}_0\r)=\min \l[\hat{\sigma}_+(u)\,:\, u \in \W\r].
	\end{align*}
	Because of \eqref{11} we conclude that $\tilde{u}_0=u_0$.
	
	From \eqref{10} and \eqref{6} it follows that $u_0$ is a local $C^1(\close)$-minimizer of $\sigma_+$. Then we know that 
	\begin{align}\label{12}
		u_0 \text{ is a local $\W$-minimizer of }\sigma_+,
	\end{align}
	see Fan \cite{Fan-2007} and Gasi\'nski-Papageorgiou \cite{8-Gasinski-Papageorgiou-2011}.
	
	Note that from \eqref{9} and \eqref{7} we see that we may assume that
	\begin{align}\label{13}
		K_{\sigma_+} \text{ is finite}.
	\end{align}
	Otherwise we already have an infinity of positive smooth solutions of \eqref{problem} all larger than $u_0$ and so we are done.
	
	From \eqref{12}, \eqref{13} and Theorem 5.7.6 of Papageorgiou-R\u{a}dulescu-Repov\v{s} \cite[p.\,449]{13-Papageorgiou-Radulescu-Repovs-2019} we can find $\rho \in (0,1)$ small enough such that
	\begin{align}\label{14}
		\sigma_+(u_0)<\inf \l[\sigma_+(u)\,:\,\|u-u_0\|=\rho\r]=m_+.
	\end{align}

	On account of hypothesis \textnormal{H$_1$}\textnormal{(ii)}, if $u \in \interior$, we have
	\begin{align}\label{15}
		\sigma_+(tu)\to -\infty \quad\text{as }t\to +\infty.
	\end{align}

	Moreover, hypotheses \textnormal{H$_1$}\textnormal{(ii)}, \textnormal{(iii)} and Proposition 4.1 of Gasi\'nski-Papageorgiou \cite{8-Gasinski-Papageorgiou-2011} imply that
	\begin{align}\label{16}
		\sigma_+ \text{ satisfies the $C$-condition}.
	\end{align}

	From \eqref{14}, \eqref{15} and \eqref{16} we see that we can use the mountain pass theorem and find $\hat{u} \in \W$ such that
	\begin{align}\label{17}
		\hat{u} \in K_{\sigma_+}\subseteq [u_0)\cap \interior\quad\text{and}\quad\sigma_+(u_0)<m_+\leq \sigma_+\l(\hat{u}\r),
	\end{align}
	see \eqref{9} and \eqref{14}. From \eqref{17} and \eqref{7} it follows that $\hat{u}\in\interior$ is the second positive smooth solution of problem \eqref{problem} with $u_0\leq \hat{u}$ and $\hat{u}\neq u_0$.
	
	In a similar way, starting with the Carath\'eodory function
	\begin{align*}
		g_-(x,s)=
		\begin{cases}
			f\l(x,s\r) &\text{if } s < v_0(x),\\[1ex]
			f\l(x,v_0(x)\r) & \text{if } v_0(x) \leq s
		\end{cases}
	\end{align*}
	and continuing as above, we can produce a second negative smooth solution $\hat{v} \in -\interior$ with $\hat{v} \leq v_0$ and $\hat{v} \neq v_0$.
\end{proof}

In fact we will show that problem \eqref{problem} admits extremal constant sign solutions, that is, there is a smallest positive solution $u_*\in\interior$ and a greatest negative solution $v_*\in -\interior$. In Section \ref{section_4} we will use these extremal constant sign solutions in order to prove the existence of a sign-changing solution, also called nodal solution.

Hypotheses \textnormal{H$_1$}\textnormal{(i)}, \textnormal{(iv)} imply that
\begin{align}\label{18}
	f(x,s)s \geq c_2 |s|^{\tau(x)-1}-c_3|s|^{r(x)-1}
\end{align}
for a.\,a.\,$x\in\Omega$, for all $s\in\R$ and for some $c_3>0$. This unilateral growth condition on $f(x,\cdot)$ leads to the following auxiliary Robin problem
\begin{equation}\label{19}
	\begin{aligned}
		-\Delta_{p(\cdot)}u+\xi(x)|u|^{p(x)-2}u &= c_2|u|^{\tau(x)-2}u-c_3|u|^{r(x)-2}u&&\text{in } \Omega, \\
		|\nabla u|^{p(x)-2} \nabla u \cdot \nu+\beta(x)|u|^{p(x)-2}u&=0&&\text{on }\partial\Omega.
	\end{aligned}
\end{equation}

For this problem we have the following existence and uniqueness result.

\begin{proposition}\label{prop_6}
	If hypotheses \textnormal{H$_0$} hold, then problem \eqref{19} admits a unique positive solution $\overline{u}\in\interior$ and since problem \eqref{19} is odd, $\overline{v}=-\overline{u}\in -\interior$ is the unique negative solution of \eqref{19}.
\end{proposition}

\begin{proof}
	First we show the existence of a positive smooth solution for problem \eqref{19}. To this end, we introduce the $C^1$-functional $\vartheta_+\colon\W\to\R$ defined by
	\begin{align*}
		\vartheta_+(u)=\gamma_{p(\cdot)}(u)+\into \frac{c_3}{r(x)} \l(u^+\r)^{r(x)}\diff x-\into \frac{c_2}{\tau(x)} \l(u^+\r)^{\tau(x)}\diff x
	\end{align*}
	for all $u\in \W$.
	
	Using Proposition \ref{prop_3} we have for all $\|u\| \geq 1$
	\begin{align*}
		\vartheta_+(u) \geq \hat{c} \|u\|^{p_-} -\frac{c_2}{\tau_-}\varrho_{\tau(\cdot)}(u)\geq \hat{c} \|u\|^{p_-}-c_4 \|u\|^{\tau_+}
	\end{align*}
	for some $c_4>0$, see also Proposition \ref{prop_1} and recall that $\W \hookrightarrow \Lp{\tau(\cdot)}$.
	
	Since $\tau_+<p_-$, see hypothesis \textnormal{H$_1$}\textnormal{(iv)}, we infer that $\vartheta_+\colon\W\to\R$ is coercive. Since it is also sequentially weakly lover semicontinuous, we can find $\tilde{u}\in\W$ such that
	\begin{align}\label{20}
		\vartheta_+\l(\tilde{u}\r)=\min \l[\vartheta_+(u)\,:\, u \in \W\r].
	\end{align}

	Since $\tau_+<p_-\leq p(x)<r(x)$ for all $x\in \close$, if $u\in\interior$ and $t\in (0,1)$ is sufficiently small, we have $\vartheta_+(tu)<0$. Then, due to \eqref{20}, it holds $\vartheta_+\l(\tilde{u}\r)<0=\vartheta_+(0)$ and so, $\tilde{u}\neq 0$.
	
	From \eqref{20} we know that $\vartheta_+'(\tilde{u})=0$ and so
	\begin{align}\label{21}
		\l\langle \gamma_{p(\cdot)}'\l(\tilde{u}\r),h\r\rangle =\into c_2 \l(\tilde{u}^+\r)^{\tau(x)-1}h\diff x-\into c_3 \l(\tilde{u}^+\r)^{r(x)-1}h\diff x
	\end{align}
	for all $h\in \W$. Choosing $h=-\tilde{u}^-\in\W$ in \eqref{21}, we get $\varrho_{0}(\tilde{u}^-)=0$ and so, $\tilde{u}\geq 0$ with $\tilde{u}\neq 0$, see Proposition \ref{prop_3}.
	
	Therefore, $\tilde{u}$ is a positive solution of \eqref{19} and as before, using the anisotropic regularity theory, see Winkert-Zacher \cite{18-Winkert-Zacher-2012} and Fan \cite{6-Fan-2007}, and the anisotropic maximum principle, see Zhang \cite{19-Zhang-2005}, we infer that $\tilde{u}\in\interior$.
	
	Next we show that this positive solution of \eqref{19} is unique. For this purpose, we introduce the integral functional $j_+\colon \Lp{1}\to \overline{R}=\R\cup \{+\infty\}$ defined by
	\begin{align*}
		j_+(u)=
		\begin{cases}
			\gamma_{p(\cdot)}\l(\nabla u^{\frac{1}{p_-}}\r)&\text{if }u\geq 0, u^{\frac{1}{p_-}} \in \W,\\[1ex]
			+\infty &\text{otherwise}.
		\end{cases}
	\end{align*}
	
	From Theorem 2.2 of Tak\'{a}\v{c}-Giacomoni \cite{16-Takac-Giacomoni-2020}, see also D\'{\i}az-Sa\'{a} \cite{4-Diaz-Saa-1987} for the isotropic case, we have that $j_+\colon \Lp{1}\to \R\cup \{+\infty\}$ is convex.
	
	Suppose that $\tilde{y}\in \W$ is another positive solution of problem \eqref{19}. As before, we have $\tilde{y}\in\interior$. Using Proposition 4.1.22 of Papageorgiou-R\u{a}dulescu-Repov\v{s} \cite[p.\,274]{13-Papageorgiou-Radulescu-Repovs-2019}, we see that
	\begin{align}\label{22}
		\frac{\tilde{u}}{\tilde{y}}\in\Linf \quad\text{and}\quad \frac{\tilde{y}}{\tilde{u}}\in \Linf.
	\end{align}

	Let $h=\tilde{u}^{p_-}-\tilde{y}^{p_-}\in \W$. Then, from \eqref{22} and if $|t|<1$ is small enough, we conclude that
	\begin{align*}
		\tilde{u}^{p_-}+th \in \dom j_+\quad\text{and}\quad \tilde{y}^{p_-}+th \in \dom j_+.
	\end{align*}
	Hence, on account of the convexity of $j_+$, we infer that $j_+$ is Gateaux differentiable at $\tilde{u}^{p_-}$ and at $\tilde{y}^{p_-}$ in the direction $h$. Using Green's identity, see Tak\'{a}\v{c}-Giacomoni \cite[Remark 2.6]{16-Takac-Giacomoni-2020}, and the chain rule, we obtain
	\begin{align*}
		j_+'\l(\tilde{u}^{p_-}\r)(h)
		&= \frac{1}{p_-} \into \frac{-\Delta_{p(\cdot)}\tilde{u}+\xi(x)\tilde{u}^{p(x)-1}}{\tilde{u}^{p_--1}}h\diff x\\
		& = \frac{1}{p_-} \into \l[\frac{c_2}{\tilde{u}^{p_--\tau(x)}}-c_3 \tilde{u}^{r(x)-p_-}\r]h\diff x
	\end{align*}
	and
	\begin{align*}
		j_+'\l(\tilde{y}^{p_-}\r)(h)
		&= \frac{1}{p_-} \into \frac{-\Delta_{p(\cdot)}\tilde{y}+\xi(x)\tilde{y}^{p(x)-1}}{\tilde{y}^{p_--1}}h\diff x\\
		& = \frac{1}{p_-} \into \l[\frac{c_2}{\tilde{y}^{p_--\tau(x)}}-c_3 \tilde{y}^{r(x)-p_-}\r]h\diff x.
	\end{align*}

	The convexity of $j_+$ implies the monotonicity of $j_+'$. Therefore, we have
	\begin{align*}
		0 & \leq \into c_2 \l[\frac{1}{\tilde{u}^{p_--\tau(x)}}-\frac{1}{\tilde{y}^{p_--\tau(x)}}\r]\l(\tilde{u}^{p_-}-\tilde{y}^{p_-}\r)\diff x\\
		&\quad + \into c_3 \l[\tilde{y}^{r(x)-p_-}-\tilde{u}^{r(x)-p_-}\r]\l(\tilde{u}^{p_-}-\tilde{y}^{p_-}\r)\diff x.
	\end{align*}
	Recall that $\tau_+<p_-<\tau(x)$ for all $x \in \close$, we conclude that $\tilde{u}=\tilde{y}$. This proves the uniqueness of the positive solution $\tilde{u}\in\interior$ for problem \eqref{19}.
	
	Since the problem is odd, $\tilde{v}=-\tilde{u} \in-\interior$ is the unique negative solution of \eqref{19}.
\end{proof}

We introduce the following two sets
\begin{align*}
	\mathcal{S}_+&=\left\{u: u\text{ is a positive solution of problem \eqref{problem}}\right\},\\
	\mathcal{S}_-&=\left\{u: u\text{ is a negative solution of problem \eqref{problem}}\right\}.
\end{align*}

We have already seen in Proposition \ref{prop_4} that 
\begin{align*}
	\emptyset \neq \mathcal{S}_+\subseteq \interior \quad\text{and}\quad 
	\emptyset\neq \mathcal{S}_- \subseteq -\interior.
\end{align*}

The solutions $\tilde{u}, \tilde{v}$ of \eqref{19} provide bounds for the sets $\mathcal{S}_+$ and $\mathcal{S}_-$, where $\tilde{u}\in\interior$ is a lower bound for $\mathcal{S}_+$ and $\tilde{v}\in-\interior$ is an upper bound for $\mathcal{S}_-$.

\begin{proposition}\label{prop_7}
	If hypotheses \textnormal{H$_0$}, \textnormal{H$_1$} hold, then $\tilde{u}\leq u$ for all $u \in \mathcal{S}_+$ and $v\leq \tilde{v}$ for all $v \in \mathcal{S}_-$.
\end{proposition}

\begin{proof}
	Let $u \in \mathcal{S}_+$ and consider the Carath\'eodory function $k_+\colon \Omega\times\R\to\R$ defined by
	\begin{align}\label{23}
		k_+(x,s)=
		\begin{cases}
			c_2 \l(s^+\r)^{\tau(x)-1}-c_3 \l(s^+\r)^{r(x)-1} &\text{if } s \leq u(x),\\[1ex]
			c_2 \l(u(x)\r)^{\tau(x)-1}-c_3 \l(u(x)\r)^{r(x)-1}  & \text{if } u(x) <s.
		\end{cases}
	\end{align}
	We set $K_+(x,s)=\int^s_0k_+(x,t)\diff t$ and consider the $C^1$-functional $\hat{\vartheta}_+\colon\W\to\R$ defined by
	\begin{align*}
		\hat{\vartheta}_+(u)=\gamma_{p(\cdot)}(u)-\into K_+(x,u)\diff x\quad\text{for all }u\in\W.
	\end{align*}
	Evidently, $\hat{\vartheta}_+\colon\W\to\R$ is coercive, see \eqref{23} and Proposition \ref{prop_3}, and sequentially weakly lower semicontinuous. Hence, we find $\tilde{u}_*\in\W$ such that
	\begin{align}\label{24}
		\hat{\vartheta}_+\l(\tilde{u}_*\r)=\min \l[\hat{\vartheta}_+(u)\,:\, u \in \W\r].
	\end{align}
	
	As before, see the proof of Proposition \ref{prop_6}, if $w \in \interior$ and $t \in (0,1)$ sufficiently small, at least so that $tw\leq u$ we have $\hat{\vartheta}_+(tw)<0$, recall that $u \in \interior$ and use Proposition 4.1.22 of Papageorgiou-R\u{a}dulescu-Repov\v{s} \cite[p.\,274]{13-Papageorgiou-Radulescu-Repovs-2019}. Then, due to \eqref{24}, it follows that $\hat{\vartheta}_+(\tilde{u}_*)<0=\hat{\vartheta}_+(0)$. Hence, $\tilde{u}_*\neq 0$.
	
	From \eqref{24} we have $(\hat{\vartheta}_+)'(\tilde{u}_*)=0$, that is,
	\begin{align}\label{25}
		\l\langle \gamma_{p(\cdot)}'\l(\tilde{u}_*\r),h\r \rangle=\into k_+\l(x,\tilde{u}_*\r)h\diff x\quad\text{for all }h\in \W.
	\end{align}

	First we choose $h=-\tilde{u}^-\in\W$ and obtain $\tilde{u}_*\geq 0$ with $\tilde{u}_*\neq 0$, see \eqref{23}. Next, we take $h=\l(\tilde{u}_*-u\r)^+\in\W$. Then, from \eqref{23}, \eqref{18} and the fact that $u \in\mathcal{S}_+$, we obtain
	\begin{align*}
		\l\langle \gamma_{p(\cdot)}'\l(\tilde{u}_*\r),\l(\tilde{u}_*-u\r)^+\r \rangle
		&=\into \l[c_2 u^{\tau(x)-1}-c_3u^{r(x)-1}\r]\l(\tilde{u}_*-u\r)^+\diff x\\
		&\leq \into f(x,u)\l(\tilde{u}_*-u\r)^+\diff x\\
		& =\l\langle \gamma_{p(\cdot)}'\l(u\r),\l(\tilde{u}_*-u\r)^+\r \rangle.
	\end{align*}
	Hence, $\tilde{u}_*\leq u$. So, we have proved
	\begin{align}\label{26}
		\tilde{u}_* \in [0,u], \ \tilde{u}_*\neq 0.
	\end{align}

	From \eqref{26}, \eqref{23}, \eqref{25} and Proposition \ref{prop_6}, it follows that $\tilde{u}_*=\tilde{u}$. Thus, see \eqref{26}, $\tilde{u}\leq u$ for all $u \in \mathcal{S}_+$.
	
	Similarly, we show that $v\leq \tilde{v}$ for all $v\in\mathcal{S}_-$.
\end{proof}

Now we are ready to produce extremal constant sign solutions for problem \eqref{problem}.

\begin{proposition}\label{prop_8}
	If hypotheses \textnormal{H$_0$}, \textnormal{H$_1$} hold, then problem \eqref{problem} has a smallest positive solution $u_* \in \interior$ and a greatest negative solution $v_*\in -\interior$. 
\end{proposition}

\begin{proof}
	From the proof of Proposition 7 of Papageorgiou-R\u{a}dulescu-Repov\v{s} \cite{11-Papageorgiou-Radulescu-Repovs-2017}, we know that $\mathcal{S}_+$ is downward directed, that is, if $u_1, u_2\in \mathcal{S}_+$, then we can find $u \in \mathcal{S}_+$ such that $u\leq u_1$ and $u\leq u_2$. Then Lemma 3.10 of Hu-Papageorgiou \cite[p.\,178]{9-Hu-Papageorgiou-1997} implies that there exists a decreasing sequence $\{u_n\}_{n\in\N} \subseteq \mathcal{S}_+$ such that 
	\begin{align*}
		\inf \mathcal{S}_+ =\inf_{n\in\N} u_n.
	\end{align*}

	On account of Proposition \ref{prop_4} we have that $\{u_n\}_{n\in\N}\subseteq \W$ is bounded. Hence, we may assume that
	\begin{align}\label{27}
		u_n\weak u_* \quad\text{in }\W \quad\text{and}\quad
		u_n\to u_* \quad\text{in }\Lp{p(\cdot)}\text{ and in }\Lprand{p(\cdot)}.
	\end{align}

	Since $u_n \in \mathcal{S}_+$, we have
	\begin{align}\label{28}
		\l\langle \gamma_{p(\cdot)}'\l(u_n\r),h\r\rangle =\into f(x,u_n)h\diff x
	\end{align}
	for all $h\in \W$ and for all $n \in \N$. Choosing $h=u_n-u_*\in \W$ in \eqref{28}, passing to the limit as $n\to \infty$ and using \eqref{27}, we obtain
	\begin{align*}
		\lim_{n\to\infty} \l\langle A_{p(\cdot)}(u_n),u_n-u_*\r\rangle=0.
	\end{align*}
	Then, from Proposition \ref{prop_2}, we infer that 
	\begin{align}\label{29}
		u_n\to u_*\quad\text{in }\W.
	\end{align}

	So, passing to the limit as $n\to\infty$ in \eqref{28} and using \eqref{29}, one gets
	\begin{align*}
		\l\langle \gamma_{p(\cdot)}'(u_*),h\r\rangle =\into f\l(x,u_*\r)h\diff x\quad\text{for all }h \in \W.
	\end{align*}

	Furthermore, by Proposition \ref{prop_7}, we conclude that $\tilde{u}_*\leq u_*$. It follows that $u_*\in\mathcal{S}_+$ and $u_*=\inf\mathcal{S}_+$. 
	
	Similarly, we show that there exists $v_*\in \mathcal{S}_-$ such that $v\leq v_*$ for all $v\in \mathcal{S}_-$. We mention that $\mathcal{S}_-$ is upward directed, that is, if $v_1, v_2\in \mathcal{S}_-$, we can find $v\in \mathcal{S}_-$ such that $v_1\leq v$ and $v_2\leq v$.
\end{proof}

\section{Nodal solution}\label{section_4}

In this section, using the extremal constant sign solutions of problem \eqref{problem} obtained in Proposition \ref{prop_8}, we show the existence of a nodal solution located between them.

\begin{proposition}\label{prop_9}
	If hypotheses \textnormal{H$_0$}, \textnormal{H$_1$} hold, then problem \eqref{problem} admits a nodal solution
	\begin{align*}
		y_0 \in \l[v_*,u_*\r]\cap C^1(\close).
	\end{align*}
\end{proposition}

\begin{proof}
	Let $u_*\in\interior$ and $v_*\in-\interior$ be the two extremal constant sign solutions produced in Proposition \ref{prop_8}. We introduce the Carath\'eodory function $l\colon\Omega\times\R\to\R$ defined by
	\begin{align}\label{30}
		l(x,s)=
		\begin{cases}
			f\l(x,v_*(x)\r)&\text{if } s <v_*(x),\\[1ex]
			f\l(x,s\r)&\text{if } v_*(x)\leq s\leq u_*(x),\\[1ex]
			f\l(x,u_*(x)\r)&\text{if } u_*(x)<s.
		\end{cases}
	\end{align}
	We also consider the positive and negative truncations of $l(x,\cdot)$, namely the Cara\-th\'eodory functions $l_{\pm}\colon\Omega\times\R\to\R$ defined by
	\begin{align}\label{31}
		l_{\pm}(x,s)=l\l(x,\pm s^{\pm}\r).
	\end{align}
	We set $L(x,s)=\int^s_0l(x,t)\diff t$ and $L_{\pm}(x,s)=\int^s_0 l_{\pm}(x,t)\diff t$ and consider the $C^1$-functional $\mu,\mu_{\pm}\colon \W\to\R$ defined by
	\begin{align*}
		\mu(u)=\gamma_{p(\cdot)}(u)-\into L(x,u)\diff x \quad\text{for all }u \in \W
	\end{align*}
	and
	\begin{align*}
		\mu_{\pm}(u)=\gamma_{p(\cdot)}(u)-\into L_{\pm}(x,u)\diff x \quad\text{for all }u \in \W.
	\end{align*}
	
	Using \eqref{30} and \eqref{31} we easily show that
	\begin{align*}
		& K_\mu\subseteq \l[v_*,u_*\r]\cap C^1(\close),\\ 
		&K_{\mu_+}\subseteq \l[0,u_*\r]\cap C^1(\close)_+\\
		&K_{\mu_-}\subseteq \l[v_*,0\r]\cap \l(-C^1(\close)_+\r).
	\end{align*}

	The extremality of the solutions $u_*$ and $v_*$ implies that
	\begin{align}\label{32}
		& K_\mu\subseteq \l[v_*,u_*\r]\cap C^1(\close),\quad K_{\mu_+}=\l\{0,u_*\r\}, \quad K_{\mu_-}= \l\{v_*,0\r\}.
	\end{align}

	Due to \eqref{30} and \eqref{31} it is clear that $\mu_+\colon \W\to\R$ is coercive and it is also sequentially weakly lower semicontinuous. Hence, $\hat{u}_*\in\W$ exists such that
	\begin{align*}
		\mu_+\l(\hat{u}_*\r)=\min \l[\mu_+(u)\,:\, u \in \W\r]<0=\mu_+(0),
	\end{align*}
	see the proof of Proposition \ref{prop_6}. Hence, $\hat{u}_*\neq 0$ and so, $\hat{u}_*=u_*$, see \eqref{32}.
	
	It is clear that
	\begin{align*}
		\mu\big|_{C^1(\close)_+}=\mu_+\big|_{C^1(\close)_+}.
	\end{align*}
	Since $u_*\in\interior$, it follows that $u_*$ is a local $C^1(\close)$-minimizer of $\mu$. Therefore,
	\begin{align}\label{33}
		u_* \text{ is a local $\W$-minimizer of }\mu,
	\end{align}
	see Fan \cite{Fan-2007} and Gasi\'nski-Papageorgiou \cite{8-Gasinski-Papageorgiou-2011}.
	
	Similarly, working this time with the functional $\mu_-$, we show that
	\begin{align}\label{34}
		v_* \text{ is a local $\W$-minimizer of }\mu.
	\end{align}

	We may assume that $\mu(v_*)\leq \mu(u_*)$. The reasoning is similar if the opposite inequality holds using \eqref{34} instead of \eqref{33}. From \eqref{32} it is clear that we may assume that $K_\mu$ is finite. Otherwise, taking \eqref{30} and the extremality of the solutions $u_*$ and $v_*$ into account, we already have an infinity of smooth nodal solutions and so we are done. Then, \eqref{33} and Theorem 5.7.6 of Papageorgiou-R\u{a}dulescu-Repov\v{s} \cite[p.\,449]{13-Papageorgiou-Radulescu-Repovs-2019}, we know that we can find $\rho\in(0,1)$ small enough such that
	\begin{align}\label{35}
		\mu(v_*)\leq \mu(u_*)<\inf \l[\mu(u)\,:\,\|u-u_*\|=\rho\r]=m_*.
	\end{align}
	The coercivity of $\mu$ implies that $\mu$ satisfies the $C$-condition, see Papageorgiou-R\u{a}dulescu-Repov\v{s} \cite[Proposition 5.1.15 on p.\,369]{13-Papageorgiou-Radulescu-Repovs-2019}. This fact along with \eqref{35} permit the use of the mountain pass theorem. So, there exists $y_0\in\W$ such that
	\begin{align}\label{36}
		y_0\in K_{\mu}\subseteq \l[v_*,u_*\r]\cap C^1(\close), \quad m_*\leq \mu(y_0).
	\end{align}
	From \eqref{35} and \eqref{36} it follows that $y_0\not\in \{v_*,u_*\}$. Moreover, from Corollary 6.6.9 of Papageorgiou-R\u{a}dulescu-Repov\v{s} \cite[p.\,533]{13-Papageorgiou-Radulescu-Repovs-2019} we have
	\begin{align}\label{37}
		C_1\l(\mu,y_0\r)\neq 0.
	\end{align}

	On the other hand, from hypothesis \textnormal{H$_1$}\textnormal{(iv)} and Proposition 3.7 of Papageorgiou-R\u{a}dulescu \cite{10-Papageorgiou-Radulescu-2016}, we obtain
	\begin{align}\label{38}
		C_k\l(\mu,0\r)=0\quad\text{for all }k\in\N_0.
	\end{align}
	Comparing \eqref{37} and \eqref{38}, we conclude that $y_0\neq 0$. Since $y_0\in [v_*,u_*]\cap C^1(\close)$ with $y_0 \not\in \{0,u_*,v_*\}$, the extremality of $u_*$ and $v_*$ implies that $y_0$ is a smooth nodal solution of \eqref{problem}.
\end{proof}

Finally, we can state the following multiplicity theorem for problem \eqref{problem}, see Propositions \ref{prop_4}, \ref{prop_5} and \ref{prop_9}.

\begin{theorem}\label{main_result}
	If hypotheses \textnormal{H$_0$}, \textnormal{H$_1$} hold, then  problem \eqref{problem} has at least five nontrivial smooth solutions
	\begin{align*}
		u_0, \hat{u} \in\interior 
		\quad\text{and}\quad 
		v_0,\hat{v}\in-\interior  \quad\text{and}\quad y_0\in C^1(\close) \text{ nodal}
	\end{align*}
	with $u_0\neq \hat{u}$, $v_0\neq \hat{v}$ and
	\begin{align*}
		\hat{v}(x) \leq v_0(x)\leq y_0(x) \leq u_0(x) \leq \hat{u}(x)\quad\text{for all }x\in\close
	\end{align*}
	as well as
	\begin{align*}
		\eta_-<v_0(x)<0<u_0(x)<\eta_+\quad
		\text{for all }x \in\Omega.
	\end{align*}
\end{theorem}


\begin{thebibliography}{99}

\bibitem{Deng-2008}
	S.-G.~Deng,
	{\it Eigenvalues of the {$p(x)$}-{L}aplacian {S}teklov problem},
	J. Math. Anal. Appl. {\bf 339} (2008), no. 2, 925--937.

\bibitem{1-Deng-2009}
	S.-G.~Deng,
	{\it Positive solutions for {R}obin problem involving the {$p(x)$}-{L}aplacian},
	J. Math. Anal. Appl. {\bf 360} (2009), no. 2, 548--560.

\bibitem{2-Deng-Wang-2010}
	S.-G.~Deng, Q.~Wang,
	{\it Nonexistence, existence and multiplicity of positive solutions to the {$p(x)$}-{L}aplacian nonlinear {N}eumann boundary value problem},
	Nonlinear Anal. {\bf 73} (2010), no. 7,  2170--2183.

\bibitem{Deng-Wang-Cheng-2011}
	S.-G.~Deng, Q.~Wang, S.~Cheng,
	{\it On the {$p(x)$}-{L}aplacian {R}obin eigenvalue problem},
	Appl. Math. Comput. {\bf 217} (2011), no. 12, 5643--5649.

\bibitem{4-Diaz-Saa-1987}
	J.I.~D\'{\i}az, J.E.~Sa\'{a},
	{\it Existence et unicit\'{e} de solutions positives pour certaines \'{e}quations elliptiques quasilin\'{e}aires},
	C. R. Acad. Sci. Paris S\'{e}r. I Math. {\bf 305} (1987), no. 12, 521--524.

\bibitem{3-Diening-Harjulehto-Hasto-Ruzicka-2011}
	L.~Diening, P.~Harjulehto, P.~H\"{a}st\"{o}, M.~R$\mathring{\text{u}}$\v{z}i\v{c}ka,
	``Lebesgue and Sobolev Spaces with Variable Exponents'',
	Springer, Heidelberg, 2011.

\bibitem{5-Fan-2008}
	X.~Fan,
	{\it Boundary trace embedding theorems for variable exponent {S}obolev spaces},
	J. Math. Anal. Appl. {\bf 339} (2008), no. 2,  1395--1412.

\bibitem{6-Fan-2007}
	X.~Fan,
	{\it Global {$C^{1,\alpha}$} regularity for variable exponent elliptic equations in divergence form},
	J. Differential Equations {\bf 235} (2007), no. 2, 397--417.

\bibitem{Fan-2007}
	X.~Fan,
	{\it On the sub-supersolution method for {$p(x)$}-{L}aplacian equations},
	J. Math. Anal. Appl. {\bf 330} (2007), no. 1, 665--682.

\bibitem{7-Fan-Deng-2009}
	X.~Fan, S.-G.~Deng,
	{\it Multiplicity of positive solutions for a class of inhomogeneous {N}eumann problems involving the {$p(x)$}-{L}aplacian},
	NoDEA Nonlinear Differential Equations Appl. {\bf 16} (2009), no. 2,  255--271.
	
\bibitem{8-Gasinski-Papageorgiou-2011}
	L.~Gasi\'{n}ski, N.S.~Papageorgiou,
	{\it Anisotropic nonlinear {N}eumann problems},
	Calc. Var. Partial Differential Equations {\bf 42} (2011), no. 3-4, 323--354.

\bibitem{Ho-Kim-Winkert-Zhang-2021}
	K.~Ho, Y.-H.~Kim, P.~Winkert, C.~Zhang,
	{\it The boundedness and H\"older continuity of solutions to elliptic equations involving variable exponents and critical growth},
	preprint, 2020, arXiv: 2010.15098.

\bibitem{9-Hu-Papageorgiou-1997}
	S.~Hu, N.S.~Papageorgiou,
	``Handbook of Multivalued Analysis'', {V}ol. {I}, Kluwer Academic Publishers, Dordrecht, 1997.

\bibitem{10-Papageorgiou-Radulescu-2016}
	N.S.~Papageorgiou, V.D.~R\u{a}dulescu,
	{\it Coercive and noncoercive nonlinear {N}eumann problems with indefinite potential},
	Forum Math. {\bf 28} (2016), no. 3, 545--571.

\bibitem{12-Papageorgiou-Radulescu-Repovs-2020}
	N.S.~Papageorgiou, V.D.~R\u{a}dulescu, D.D.~Repov\v{s},
	{\it Anisotropic equations with indefinite potential and competing nonlinearities},
	Nonlinear Anal. {\bf 201} (2020), 111861, 24 pp.

\bibitem{13-Papageorgiou-Radulescu-Repovs-2019}
	N.S.~Papageorgiou, V.D.~R\u{a}dulescu, D.D.~Repov\v{s},
	``Nonlinear Analysis --- Theory and Methods'',
	Springer, Cham, 2019.

\bibitem{11-Papageorgiou-Radulescu-Repovs-2017}
	N.S.~Papageorgiou, V.D.~R\u{a}dulescu, D.D.~Repov\v{s},
	{\it Positive solutions for perturbations of the {R}obin eigenvalue problem plus an indefinite potential},
	Discrete Contin. Dyn. Syst. {\bf 37} (2017), no. 5, 2589--2618.

\bibitem{14-Papageorgiou-Radulescu-Tang-2021}
	N.S.~Papageorgiou, V.D.~R\u{a}dulescu, X.~Tang,
	{\it Anisotropic Robin problems with logistic reaction},
	Z. Angew. Math. Phys., to appear.

\bibitem{15-Radulescu-Repovs-2015}
	V.D.~R\u{a}dulescu, D.D.~Repov\v{s},
	``Partial Differential Equations with Variable Exponents'',
	CRC Press, Boca Raton, FL, 2015.

\bibitem{16-Takac-Giacomoni-2020}
	P.~Tak\'{a}\v{c}, J.~Giacomoni,
	{\it A {$p(x)$}-{L}aplacian extension of the {D}\'{\i}az-{S}aa inequality and some applications},
	Proc. Roy. Soc. Edinburgh Sect. A {\bf 150} (2020), no. 1, 205--232.

\bibitem{17-Wang-Fan-Ge-2009}
	L.-L.~Wang, Y.-H.~Fan, W.-G.~Ge,
	{\it Existence and multiplicity of solutions for a {N}eumann problem involving the {$p(x)$}-{L}aplace operator},
	Nonlinear Anal. {\bf 71} (2009), no. 9, 4259--4270.

\bibitem{18-Winkert-Zacher-2012}
	P.~Winkert, R.~Zacher,
	{\it A priori bounds for weak solutions to elliptic equations with nonstandard growth},
	Discrete Contin. Dyn. Syst. Ser. S {\bf 5} (2012), no. 4, 865--878.

\bibitem{19-Zhang-2005}
	Q.~Zhang,
	{\it A strong maximum principle for differential equations with nonstandard {$p(x)$}-growth conditions},
	J. Math. Anal. Appl. {\bf 312} (2005), no. 1, 24--32.

\end{thebibliography}
\end{document}